\documentclass[11pt]{amsart}
\usepackage{amsmath,amssymb, amsthm}
\usepackage{pdfsync}
\usepackage{color}
\usepackage{comment}

\newcommand{\F}{\mathbb F}
\newcommand\Z{{\mathbb Z}}
\newcommand\R{{\mathbb R}}

\newcommand\N{{\mathbb N}}

\newcommand\inv{^{-1}}

\newcommand{\Ac}{\mathcal A}
\newcommand{\CC}{\mathcal C}
\newcommand{\GG}{\mathcal G}
\newcommand{\Ga}{\Gamma}
\newcommand{\ga}{\gamma}
\renewcommand{\H}{\mathbf{H}}
\newcommand{\gl}{\textnormal{GL}}
\newcommand{\ut}{\textnormal{UT}}

\newtheorem{theorem}{Theorem}[section]

\newtheorem{lemma}[theorem]{Lemma}

\newtheorem{corollary}[theorem]{Corollary}

\begin{document}
\title[Cross-wired lamplighters and linearity of automata groups]{Cross-wired Lamplighter Groups and Linearity of Automata Groups}
\author{Ning Yang}

\maketitle
\begin{abstract}We consider the two generalizations of lamplighter groups: automata groups generated by Cayley machine and cross-wired lamplighter groups. For a finite step two nilpotent group with central squares, we study its associated Cayley machine and give a presentation of the corresponding automata group. We show the automata group is a cross-wired lamplighter group and does not embed in the wreath product of a finite group with a torsion free group. For a subfamily of such finite step two nilpotent groups, we prove that their associated automata groups are linear.
\end{abstract}

\section{Introduction}

The so called lamplighter group is a popular object of study. It is defined as a restricted wreath product $\Z_2\wr\Z$, namely a semi-direct product $(\oplus_{\Z}\Z_2) \rtimes \Z$ with the action of $\Z$ on $\oplus_{\Z}\Z_2$ by shifts, where $\Z$ denotes integers and $\Z_2$ denotes $\Z/2\Z$. It is an infinitely presented 2-generated group. It is step-2 solvable and has exponential growth. In general, a lamplighter group could be a group of the form $F \wr\Z$ where $F$ is some nontrivial finite group. There are two kinds of generalization of lamplighter groups: automata groups generated by Cayley machines \cite{GZ,SiSt} and cross-wired lamplighter groups \cite{CFK}.

Grigorchuk and \.{Z}uk \cite{GZ} showed that the lamplighter group $\Z_2\wr\Z$ can be constructed as the automata group of a 2-state automaton. They used this automaton to compute the spectrum and spectral measure associated to random walks on the group, leading to a counterexample of the strong form of the Atiyah conjecture. Silva and Steinberg \cite{SiSt} showed that, for a finite abelian group $F$, the lamplighter group $F \wr\Z$ can be generated by a reset automaton, which they called the Cayley machine of the group $F$.  We study the automata groups generated by the Cayley machines of some finite nilpotent groups.

\begin{theorem}\label{pres}
Let $G=\{g_1=1,g_2,\ldots,g_k\}$ be a step-2 nilpotent group. Assume the square of each element of $G$ is central, then the automata group generated by its Cayley machine has the following presentation: $$\langle x,g_2,\ldots,g_k  \ |\  [ x^mg_jx^{-m}, \ x^ng_ix^{-n}]=x^{[\frac{n+m}{2}]}[g_j,g_i]x^{-[\frac{n+m}{2}]}, m,n\in\Z, 2\leq j,i\leq k\rangle,$$ where $[a]$ means the least integer greater than or equal to $a\in \R$ and $[g_j ,g_i ]$ denotes group commutator.
\end{theorem}

\noindent\textbf{Remark:}
Floriane Pochon, a student of L. Bartholdi, found similar results for the case of the dihedral group of order $8$ around 2005. Her paper is entitled \textit{Structure De La Machine De Cayley De $D_8$} and never published.

Thus, the theorem of Silva and Steinberg for the case of finite abelian groups is a corollary of Theorem \ref{pres}. For a  finite non-abelian group $F$,  $F \wr\Z$ can not be a automata group. Indeed, for such finite group $F$, $F \wr\Z$ is residually finite if and only if $F$ is abelian \cite{Gru}. On the other hand, automata groups are always residually finite \cite{GNS}. Silva and Steinberg conjectured that in the case of finite non-abelian group, the resulting automata group does not embed in the wreath product of a finite group with a torsion free group. In the appendix of \cite{KSS}, Kambites, Silva and Steinberg partially proved the conjecture. They showed  for any finite non-abelian group that is not a direct product of an abelian group with a 2-group which is nilpotent of class 2, its associated automata group can not embed in a wreath product of a finite group with a torsion free group. We show the case when it contains a 2-group which is nilpotent of class 2.

\begin{theorem}\label{conje}
Let $G$ be a finite non-abelian group containing a 2-group that is nilpotent of class 2. Then the automata group generated by its Cayley machine does not embed in a wreath product of a finite group with a torsion free group.
\end{theorem}

Combining results in the appendix of \cite{KSS}, we have

\begin{theorem}\label{embed}
The automata group associated  to any finite non-abelian group does not embed in a wreath product of a finite group with a torsion free group.
\end{theorem}

Some Cayley graphs of lamplighter groups are examples of Diestel-Leader graphs \cite{DL, Moeller, Woess, Wortman}. Eskin, Fisher, and Whyte proved \cite{EFW1, EFW2, EFW3} that a finitely generated group is quasi-isometric to a
lamplighter group if and only if it acts properly and cocompactly by
isometries on some Diestel-Leader graph.  
 In \cite{CFK}, Cornulier, Fisher and Kashyap studied quasi-isometric rigidity question of lamplighter groups by studying the cocompact lattices in the isometry group of the Diestel-Leader graphs. They showed that these lattices are not necessarily lamplighters and gave an algebraic characterization of them. They call these cocompact lattices \textit{cross-wired lamplighter groups}. Lamplighter groups are examples of cross-wired lamplighters. 
 It is suggested by L. Bartholdi that automata groups associated to some finite groups might provide other examples of cross-wired lamplighter groups. We show that the automata groups in Theorem \ref{pres} are cross-wired lamplighter groups.

\begin{theorem}\label{lamp}
The automata groups in Theorem \ref{pres} are cross-wired lamplighter groups.
\end{theorem}

Cornulier, Fisher and Kashyap also give interesting examples that are not virtually lamplighter groups \cite{CFK}. Let $q$ be a prime power, and consider the Laurent polynomial ring $\F_q[t^{\pm 1}]$. 
Let $\H$ be the Heisenberg group, consisting of unitriangular $3\times 3$ matrices. Consider the action of $\Z$ on $\H(\F_q[t^{\pm 1}])$ defined by the automorphism
\[
\Phi: \ \left(
  \begin{array}{ccc}
  1 & x & z\\
  0 & 1 & y\\
  0 & 0 & 1
  \end{array}
  \right) \mapsto
\left(
  \begin{array}{ccc}
  1 & tx & t^2z\\
  0 & 1 & ty\\
  0 & 0 & 1
  \end{array}
  \right).
\]
Then the group $\H(\F_q[t^{\pm 1}])\rtimes_\Phi \Z$ is a cross-wired lamplighter group. This construction applies to other nilpotent groups over local field of positive characteristic with similar contraction $\Z$ actions. In the beginning, we believed that the automata groups in Theorem \ref{pres} might be totally different examples of cross-wired lamplighters from those examples above, or even nonlinear. By \textit{linear} we mean the group can be embedded in a general linear group over some fields. As we know, many automata groups are nonlinear. The Grigorchuk groups \cite{G, E} are automata groups, of intermediate growth and hence nonlinear. But the lamplighter group $\Z_2\wr\Z$ is linear over function fields \cite{Wortman}. So it is interesting to study the linearity/nonlinearity of those automata groups above. We find the following surprising. Let $Q_8$ be the quaternion group of order $8$ and $M_{2^n}$ be the Iwasawa/modular group of order $2^n$, $n\geq 3$, $M_{2^n}=\langle a, b \ |\ a^{2^{n-1}}=b^2=e,  bab\inv=a^{2^{n-2}+1} \rangle.$ The automata groups associated to $Q_8$ and $M_{2^n}$ have index two subgroups behave like above linear examples with diagonal contracting $\Z$ actions.

\begin{theorem}\label{linear}
The automata groups associated to $Q_8$ and $M_{2^n}$ are linear over $\Z_2[t^{\pm 1}]$.
\end{theorem}

\medskip

\noindent\textbf{Acknowledgements.}  The author would like to express his deep gratitude to his advisor, Professor David Fisher, for consistent support, enlightening guidance and encouragement. He also thanks Laurent Bartholdi and Michael Larsen for their interests and inspiring discussions. He thanks Yves de Cornulier for corrections on an early draft of this paper.

\section{Automata groups}

\subsection{Automata and Cayley Machine} This subsection is similar with part of \cite{KSS}.

A $finite \ (Mealy) \  automaton$ \cite{Eilenberg, GNS,KSS, SiSt}
$\Ac$ is a $4$-tuple $(Q,A,\delta,\lambda)$, where $Q$ is a finite
set of states, $A$ is a finite alphabet, $\delta:Q\times A\to Q$
is the transition function and $\lambda: Q\times A\to A$ is the
output function. We write $qa$ for $\delta(q,a)$ and $q\circ a$
for $\lambda (q,a)$. Let $A^*$ be the set of all words in letters of $A$.
These functions extend to the free monoid $A^*$ by $$q(ab) = (qa)b,$$ \begin{equation}\label{action}
q\circ (ab) = (q\circ a)(qa)\circ b.
\end{equation}

Let $\Ac_q$ denote the 
\textit{initial automaton} $\Ac$ with designated start state $q\in Q$.
There is a function $\Ac_q:A^*\to A^*$ given by $w\mapsto q\circ
w$. This function is word length preserving and extends continuously \cite{GNS}
to the set of right infinite words $A^{\omega}$ via 
\begin{equation}\label{actiononcantor}
\Ac_q(a_0a_1\cdots) = \lim_{n \to \infty} \Ac_q(a_0\cdots a_n),
\end{equation}
where $A^{\omega}$
is given the product topology and so homeomorphic to a Cantor
set. If the function $\lambda_q:A\to A$, defined
by $\lambda_q(a) = q\circ a$, is a permutation for each $q$, then $\Ac_q$ is an
isometry of $A^{\omega}$ with metric $d(u,v) = 1/(n+1)$, where
$n$ is the length of the longest common prefix of $u$ and $v$
\cite{GNS}. In this case the automaton is called
\emph{invertible} and its \emph{inverse} is dented by $\Ac\inv$.  Let $\Ga=\GG (\mathcal
A)$, $\Ac$ invertible, be the group generated by all $\Ac_q$'s with $q\in Q$, and it is called
the \emph{automata group generated by automaton} $\Ac$.

Let $T$ be the Cayley tree of $A^*$, then $\Ga$ acts on $T$ by rooted tree automorphisms \cite{GNS} via the action \eqref{action}.  The induced action on the boundary
$\partial T$, the space of infinite directed paths from the root,
is the action \eqref{actiononcantor} of $\Ga$ on
$A^{\omega}$.
Let $\mathrm{Aut}(T)$ be the automorphism group of $T$. It is the iterated
permutational wreath product of countably many copies of the
left permutation group $(S_{A},A)$ \cite{BGS, GNS, Rhodes},
where $S_A$ is the symmetric group on $A$. 
For a group $\Ga=\GG
(\Ac)$ generated by an automaton over $A$, one has the embedding
\begin{equation}\label{wreathembedding1}
(\Ga,A^{\omega}) \hookrightarrow (S_{|A|},A)\wr (\Ga,A^{\omega}).
\end{equation}
Here we use the notation such that the wreath product of left permutation groups has a natural projection to its leftmost factor.
Let $A=\{a_1,\ldots,a_n\}$. In wreath product coordinates, an element 
in the automata group can be represented as $\sigma(f_1,\ldots,f_n)$,
where $\sigma\in S_A$ and $f_1,\ldots,f_n\in \Ga$.
It acts on the word $a_ix_2x_3\ldots$ over $A$ by the rule 
$$\sigma(f_1,\ldots,f_n)(a_ix_2x_3\ldots)=\sigma(a_i)f_i(x_2x_3\ldots).$$
Multiplication in wreath product coordinates is given by
\begin{equation}\label{multiply}
\sigma (f_1,\ldots,f_n)\tau (g_1,\ldots,g_n)=\sigma\tau(f_{\tau(1)}g_1,
\ldots,f_{\tau(n)}g_n),
\end{equation}
where $\sigma, \tau \in S_A\cong S_{\{1,2,\dots,n\}}$ and $f_1,\ldots,f_n,g_1,\ldots,g_n\in \Ga$.
The map sends $\Ac_q$ to the element with wreath product coordinates:
\begin{equation}\label{wreathcoords}
\Ac_q = \lambda_q(\Ac_{qa_1},\ldots,\Ac_{qa_n}).
\end{equation}
 Its inverse is given by
$$\Ac\inv_q=\lambda\inv_q(\Ac\inv_{a_1},\ldots,\Ac\inv_{a_n}).$$

\medskip

Let $G=\{g_1=1,g_2,\ldots,g_n\}$ be a non-trivial finite group. By the 
\textit{Cayley machine} $\mathcal C(G)$ of $G$ we mean the 
automaton with state $G$ and
alphabet $G$. Both the transition and the output functions are the
group multiplication, i.e., at state $g_0$ on input $g$ the
machine goes to state $g_0g$ and outputs $g_0g$. 
 The state function $\lambda_g$ is just left translation by $g$ and hence a
permutation, so $\CC (G)$ is invertible. The study of the automata group of the Cayley machine of a finite group was initiated by Silva and Steinberg \cite{SiSt}.

An automaton is called a \emph{reset automaton} if, for each $a\in
A$, $|Qa|=1$; that is, each input resets the automaton to a single
state.  Silva and Steinberg \cite{SiSt} showed that the inverse of
a state $\CC(G)_g, g\in G,$ is computed by the corresponding state of the
reset automaton $\mathcal A(G)$ with states and input alphabet
$G$, where at state $g_0$ on input $g$ the automaton goes to
state $g$ and outputs $g_0\inv g$.  Therefore $$\CC(G)\inv_g=\Ac(G)_g,
\ \  \ \GG(\CC(G))=\GG(\Ac(G)).$$ In this case, we have the embedding
\begin{equation}\label{wreathembedding}
\GG(\Ac(G))\hookrightarrow (G,G)\wr (\GG (\Ac (G)),G^{\omega}).
\end{equation}
In wreath product coordinates,
\begin{equation}\label{wreathcoords2}
\Ac(G)_g = g\inv(\Ac_{g_1},\ldots,\Ac_{g_n}).
\end{equation}

 Let $x:= \Ac(G)_1=\CC(G)_1^{-1}$.  Notice that \cite{SiSt}
 \begin{align*}
 x\Ac(G)\inv_g &=x\CC(G)_g
 \\ &=(\Ac_{g_1},\ldots,\Ac_{g_n})g(\CC(G)_{gg_1}\ldots,\CC(G)_{gg_n})
 \\ &=g(\Ac_{gg_1}\CC(G)_{gg_1},\ldots,\Ac_{gg_n}\CC(G)_{gg_n})
 \\ &=g(1,\ldots,1),
 \end{align*}
so we can identify $G$
with a subgroup of $\GG(\Ac(G))$ via $g\leftrightarrow
x\Ac(G)\inv_g$. Recall from \cite{SiSt} Equation 4.3 or by \eqref{wreathcoords2} that
$$x(f_0, f_1,f_2,f_3,f_4,\ldots)=(f_0,f\inv_0f_1,f\inv_1f_2,f\inv_2f_3,f\inv_3f_4,\ldots),$$
$$x\inv(f_0, f_1,f_2,f_3,f_4,\ldots)=(f_0,f_0f_1,f_0f_1f_2,f_0f_1f_2f_3,f_0f_1f_2f_3f_4,\ldots),$$
where $(f_0, f_1,f_2,f_3,f_4,\ldots)\in G^{\omega}.$

Let
$$N:=\langle x^nGx^{-n}\mid n\in \Z\rangle.$$
It is shown in \cite{SiSt} that $x$ has infinite order, $N$ is a
locally finite group and $\GG(\Ac(G)) = N\rtimes\langle x\rangle$,
where $x$ acts on $N$ by conjugates.
If $G$ is abelian, it is shown in \cite{SiSt} that
$$\GG(\Ac(G)) \cong G \wr\Z.$$  That is the automata group
is a lamplighter group. If $G$ is non-abelian and not of nilpotency class 2, then
$\GG(\Ac(G))$ does not embed in the wreath product of a finite group with a torsion free group.

The \emph{depth} of an element
$\ga\in \Ga$ is the least integer $n$ (if it exists, otherwise infinity) so that
$\ga$ only changes the first $n$ letters of a word. For example, $g\in G$ has depth $1$.
Here is a useful lemma about depth:

\begin{lemma}\label{depth}\cite{SiSt}
Let $\Ac=(Q,A,\delta,\lambda)$ be an invertible reset automaton and let $a=\Ac_p, 
b=\Ac_q$. Suppose $f\in \Ga$ has depth $n$, then $afb\inv$ has depth at most $n+1$.
\end{lemma}

Let
$$N_0 := \langle x^nGx^{-n}\mid n \geq 0\rangle.$$
It is shown in \cite{SiSt} that $x^ngx^{-n}, g\neq 1$, has depth $n+1$ and
so $N_0$ consists of finitary automorphisms.   It is shown in
\cite{SiSt} that the elements of the form $\Ac(G)_g \in \Ga$ with
$g \in G$ generate a free subsemigroup of $\Ga$.

\subsection{Proof of Theorem \ref{pres}}

Let $$G=\{g_1=1,g_2,\ldots,g_k\}$$ be a group of nilpotency class two, for some $k\in \N$. Moreover we assume $g_i^2=g_ig_i$ is central in $G$ for each $i, 1\leq i\leq k$. The dihedral group $D_8$ of order 8 and the quaternion group $Q_8$ are examples of such $G$.

Note that every finite group of nilpotency class two is obtained by taking direct products of finite groups of prime power order and class two. Therefore, $G$ must be a 2-group direct product with abelian factors.

Let $\CC(G)$ be the  Cayley machine  of $G$, then $\GG(\CC(G))$ is generated by $$\{\CC(G)_{g_1},\CC(G)_{g_2},\ldots,\CC(G)_{g_k}\}.$$
Let $x=\CC(G)\inv_{g_1}.$ Then in wreath product coordinates, $$x=1\Big(\CC(G)\inv_{g_1},\CC(G)\inv_{g_2},\ldots,\CC(G)\inv_{g_k}\Big).$$It has infinite order. Identifying $G$ with a subgroup of $\CC(G)$ via 
$g\mapsto x\CC(G)_g,$ then we have 
\begin{equation}\label{Cg}
\CC(G)_g=x\inv g, \ \ \ \CC(G)\inv_g=g\inv x.
\end{equation}
Therefore, for any $g, h\in G$,
\begin{equation}\label{inv}
 \CC(G)\inv_g \CC(G)_h = g\inv x x\inv h = g\inv h.
 \end{equation}
 
From previous subsection, $x^n g x^{-n}$ has depth $n+1$ for $n\geq0$. Also, we have
\begin{align*}\label{xngxn}
x^n g x^{-n}  &=\CC(G)^{-n}_{g_1}g\CC(G)^n_{g_1}
\\ &=(\CC(G)^{-n}_{g_1},\ldots,\CC(G)^{-n}_{g_k})g(1,\ldots,1)(\CC(G)^{n}_{g_1},\ldots,\CC(G)^{n}_{g_k})
\\ &=(\CC(G)^{-n}_{g_1},\ldots,\CC(G)^{-n}_{g_k})g(\CC(G)^{n}_{g_1},\ldots,\CC(G)^{n}_{g_k})
\\ &=g\Big(\CC(G)^{-n}_{gg_1}\CC(G)^{n}_{g_1},\  \CC(G)^{-n}_{gg_2}\CC(G)^{n}_{g_2},\ \ldots,\  \CC(G)^{-n}_{gg_k}\CC(G)^{n}_{g_k}\Big)
\\ &=g\Big(\CC^{-n}_{gg_1}\CC^{n}_{g_1},\  \CC^{-n}_{gg_2}\CC^{n}_{g_2},\ \ldots,\  \CC^{-n}_{gg_k}\CC^{n}_{g_k}\Big),
\end{align*}
where in the last step, we write $\CC$ as $\CC(G)$ for simplicity. We will keep on using this notation.

Let $N=\langle x^ngx^{-n}\mid n \in \Z, g\in G\rangle,$ then $\GG(\CC(G))=N\rtimes\langle x\rangle$, where $x$ acts on $N$ by conjugates.  We show the following key lemma.

\begin{lemma}\label{keylemma}
For any $n\geq 0,$ any $f,g,h\in G$,  
$$x^n g x^{-n}h=h x^{[\frac{n}{2}]} [h,g\inv] x^{-[\frac{n}{2}]} x^n g x^{-n},$$
$$h x^n g x^{-n}= x^n g x^{-n} x^{[\frac{n}{2}]} [g,h\inv] x^{-[\frac{n}{2}]} h,$$
and $x^l [g,h] x^{-l}$ commutes with $x^m f x^{-m}$ as long as $n\geq l, m\geq0$,
where $[g,h]$ denotes the commutator $g\inv h\inv gh$.
\end{lemma}

\begin{proof}
The proof is by induction.

The case of $n=0$ is trivial: $gh=h[h,g\inv]g$.

Let $n=1$. Then\begin{align*}
x g x\inv h &= g\Big(\CC\inv_{gg_1}\CC_{g_1},\  \CC\inv_{gg_2}\CC_{g_2},\ \ldots,\  \CC\inv_{gg_k}\CC_{g_k}\Big)h(1,\ldots,1)
\\ &= gh\Big(\CC\inv_{ghg_1}\CC_{hg_1},\  \CC\inv_{ghg_2}\CC_{hg_2},\ \ldots,\  \CC\inv_{ghg_k}\CC_{hg_k}\Big),
\end{align*}
and
\begin{align*}
hx[h,g\inv]x\inv xgx\inv &=hxh\inv ghx\inv
\\ &=h(1,\ldots,1)h\inv gh\Big(\CC\inv_{h\inv ghg_1}\CC_{g_1},\  \CC\inv_{h\inv ghg_2}\CC_{g_2},\ \ldots,\  \CC\inv_{h\inv ghg_k}\CC_{g_k}\Big)
\\ &= gh\Big(\CC\inv_{h\inv ghg_1}\CC_{g_1},\  \CC\inv_{h\inv ghg_2}\CC_{g_2},\ \ldots,\  \CC\inv_{h\inv ghg_k}\CC_{g_k}\Big).
\end{align*}
Applying \eqref{Cg} and \eqref{inv} we then have, for any $f\in G$, 
$$\CC\inv_{ghf}\CC_{hf}=f\inv h\inv g\inv hf =\CC\inv_{h\inv ghf}\CC_f.$$
Thus, by comparing the wreath product coordinates, $x g x\inv h=hx[h,g\inv]x\inv xgx\inv$. Similarly we have $hxgx\inv = xgx\inv x[g,h\inv]x\inv h.$

Since $G$ is of nilpotency class two, both $[g,h]$ and $x[g,h]x\inv$ commute with $f$ and  $xfx\inv$, for any $g, h, f \in G$. 

Let $n=2$. Then we compute
\begin{align*}
x^2gx^{-2}h &= g\Big(\CC^{-2}_{gg_1}\CC^2_{g_1},\  \CC^{-2}_{gg_2}\CC^2_{g_2},\ \ldots,\  \CC^{-2}_{gg_k}\CC^2_{g_k}\Big)h(1,\ldots,1)
\\ &= gh\Big(\CC^{-2}_{ghg_1}\CC^2_{hg_1},\  \CC^{-2}_{ghg_2}\CC^2_{hg_2},\ \ldots,\  \CC^{-2}_{ghg_k}\CC^2_{hg_k}\Big),
\end{align*}
and
\begin{align*}
&hx[h,g\inv]x\inv x^2gx^{-2} \\&= h(1,\ldots,1)[h,g\inv] \Big(\CC\inv_{h\inv ghg\inv g_1}\CC_{g_1},\CC\inv_{h\inv ghg\inv g_2}\CC_{g_2},\ldots, \CC\inv_{h\inv ghg\inv g_k}\CC_{g_k}\Big)\\ & \ \ \ \ g\Big(\CC^{-2}_{gg_1}\CC^2_{g_1},\  \CC^{-2}_{gg_2}\CC^2_{g_2},\ \ldots,\  \CC^{-2}_{gg_k}\CC^2_{g_k}\Big)
\\ &=ghg\inv\Big(\CC\inv_{h\inv ghg\inv g_1}\CC_{g_1},\CC\inv_{h\inv ghg\inv g_2}\CC_{g_2},\ldots, \CC\inv_{h\inv ghg\inv g_k}\CC_{g_k}\Big)\\ & \ \ \ \ g\Big(\CC^{-2}_{gg_1}\CC^2_{g_1},\  \CC^{-2}_{gg_2}\CC^2_{g_2},\ \ldots,\  \CC^{-2}_{gg_k}\CC^2_{g_k}\Big)
\\ &=gh\Big( \CC\inv_{h\inv ghg_1}\CC_{gg_1}\CC^{-2}_{gg_1}\CC^2_{g_1},\ \CC\inv_{h\inv ghg_2}\CC_{gg_2}\CC^{-2}_{gg_2}\CC^2_{g_2},\ \ldots,\ \CC\inv_{h\inv ghg_k}\CC_{gg_k}\CC^{-2}_{gg_k}\CC^2_{g_k}\Big).
\end{align*}

For any $f\in G$, applying \eqref{inv} we compute
\begin{align*}
\CC^{-2}_{ghf}\CC^2_{hf} &= \CC\inv_{ghf} f\inv h\inv g\inv hf \CC_{hf}
\\ &= f\inv h\inv g\inv (x f\inv h\inv g\inv hf x\inv hf)
\\ &= f\inv h\inv g\inv (hf x f\inv h\inv f\inv h\inv g\inv hf hf x\inv)
\\ &= f\inv h\inv g\inv hf x g\inv x\inv,
\end{align*}
and
\begin{align*}
\CC\inv_{h\inv ghf}\CC_{gf}\CC^{-2}_{gf}\CC^2_f &= \big((h\inv ghf)\inv gf\big) \big(f\inv g\inv f x g\inv x\inv\big)
\\ &= f\inv h\inv g\inv hf x g\inv x\inv
\\ &= \CC^{-2}_{ghf}\CC^2_{hf}.
\end{align*}

So, $x^2gx^{-2}h=hx[h,g\inv]x\inv x^2gx^{-2}$, since they have same wreath product coordinates. 

On the other hand, noting that $$[h, g\inv][g,h\inv]=h\inv ghg\inv g\inv hgh\inv=h\inv g g\inv hgh\inv hg\inv=1,$$ we have
\begin{align*}
hx^2gx^{-2} &= hx^2gx^{-2} x[h, g\inv]x\inv x[g,h\inv]x\inv
\\ &= h x xgx\inv [h, g\inv] x\inv  x[g,h\inv]x\inv
\\ &= h x [h, g\inv]xgx\inv x\inv  x[g,h\inv]x\inv
\\ &= h x [h, g\inv]x\inv x^2gx^{-2} x[g,h\inv]x\inv
\\ &= x^2gx^{-2} h x[g,h\inv]x\inv
\\ &= x^2gx^{-2}x[g,h\inv]x\inv h.
\end{align*}

Since $G$ is of nilpotency class two, then $x^2[g,h]x^{-2}$ commutes with $f, x^2fx^{-2}$, for any $g, h, f \in G$. Also we have
\begin{align*}
x^2[g,h]x^{-2} xfx\inv &= x x[g,h]x\inv f x\inv
\\ &=x f x[g,h]x\inv x\inv
\\ &=xfx\inv x^2[g,h]x^{-2}.
\end{align*}
Similarly, $[g,h]$ and $x[g,h]x\inv$ commute with $f, xfx\inv$ and  $x^2fx^{-2}$, for any $g, h, f \in G$.
Then we proved the lemma for the case of $n=2$.

Now we assume the lemma is true for all $m\leq 2n$, we need to show the cases of $2n+1$ and $2n+2$.

Applying \eqref{inv}, we have

\begin{align*}
x^{2n+1}gx^{-2n-1}h &=g\Big(\CC^{-2n-1}_{gg_1}\CC^{2n+1}_{g_1},\  \CC^{-2n-1}_{gg_2}\CC^{2n+1}_{g_2},\ \ldots,\  \CC^{-2n-1}_{gg_k}\CC^{2n+1}_{g_k}\Big)h(1,\ldots,1)
\\ &=gh\Big(\CC^{-2n-1}_{ghg_1}\CC^{2n+1}_{hg_1},\  \CC^{-2n-1}_{ghg_2}\CC^{2n+1}_{hg_2},\ \ldots,\  \CC^{-2n-1}_{ghg_k}\CC^{2n+1}_{hg_k}\Big),
\end{align*}
and similarly $x^{2n}gx^{-2n}h=gh\Big(\CC^{-2n}_{ghg_1}\CC^{2n}_{hg_1},\  \CC^{-2n}_{ghg_2}\CC^{2n}_{hg_2},\ \ldots,\  \CC^{-2n}_{ghg_k}\CC^{2n}_{hg_k}\Big).$

On the other hand, we compute
\begin{align*}
&hx^{n+1}[h,g\inv]x^{-n-1}x^{2n+1}gx^{-2n-1}\\&=h[h,g\inv]\Big(\CC^{-n-1}_{[h,g\inv]g_1}\CC^{n+1}_{g_1},\  \CC^{-n-1}_{[h,g\inv]g_2}\CC^{n+1}_{g_2},\ \ldots,\  \CC^{-n-1}_{[h,g\inv]g_k}\CC^{n+1}_{g_k}\Big)\\ & \ \ \ \ g\Big(\CC^{-2n-1}_{gg_1}\CC^{2n+1}_{g_1},\  \CC^{-2n-1}_{gg_2}\CC^{2n+1}_{g_2},\ \ldots,\  \CC^{-2n-1}_{gg_k}\CC^{2n+1}_{g_k}\Big)
\\ &= gh\Big(\CC^{-n-1}_{[h,g\inv]gg_1}\CC^{n+1}_{gg_1}\CC^{-2n-1}_{gg_1}\CC^{2n+1}_{g_1},\  \CC^{-n-1}_{[h,g\inv]gg_2}\CC^{n+1}_{gg_2}\CC^{-2n-1}_{gg_2}\CC^{2n+1}_{g_2},\\ & \ \ \ \  \ldots,\  \CC^{-n-1}_{[h,g\inv]gg_k}\CC^{n+1}_{gg_k}\CC^{-2n-1}_{gg_k}\CC^{2n+1}_{g_k}\Big)
\\ &= gh\Big(\CC^{-n-1}_{[h,g\inv]gg_1}\CC^{-n}_{gg_1}\CC^{2n+1}_{g_1},\  \CC^{-n-1}_{[h,g\inv]gg_2}\CC^{-n}_{gg_2}\CC^{2n+1}_{g_2}, \ldots,\  \CC^{-n-1}_{[h,g\inv]gg_k}\CC^{-n}_{gg_k}\CC^{2n+1}_{g_k}\Big),
\end{align*}

and similarly 
$$hx^{n}[h,g\inv]x^{-n}x^{2n}gx^{-2n}=gh\Big(\CC^{-n}_{[h,g\inv]gg_1}\CC^{-n}_{gg_1}\CC^{2n}_{g_1},\  \CC^{-n}_{[h,g\inv]gg_2}\CC^{-n}_{gg_2}\CC^{2n}_{g_2}, \ldots,\  \CC^{-n}_{[h,g\inv]gg_k}\CC^{-n}_{gg_k}\CC^{2n}_{g_k}\Big).$$

Note that, by inductive hypothesis and comparing wreath product coordinates, we have, for any $f\in G$,
$$\CC^{-2n}_{ghf}\CC^{2n}_{hf} =\CC^{-n}_{[h,g\inv]gf}\CC^{-n}_{gf}\CC^{2n}_f.$$
Then, applying \eqref{Cg} and \eqref{inv}, for any $f\in G$, we compute
\begin{align*}
&\CC^{-2n-1}_{ghf}\CC^{2n+1}_{hf} = \CC\inv_{ghf}\CC^{-n}_{[h,g\inv]gf}\CC^{-n}_{gf}\CC^{2n}_f\CC_{hf}
\\ &=(ghf)\inv x \big(([h,g\inv]gf)\inv x\big)^n \big((gf)\inv x\big)^{n-1} f\inv g\inv f \big(x\inv f\big)^{2n-1}x\inv hf
\\ &= (ghf)\inv \big(\prod^{n}_{l=1}x^l([h,g\inv]gf)\inv x^{-l}\big)\big( \prod^{2n-1}_{l=n+1} x^{l}(gf)\inv x^{-l}\big) x^{2n}f\inv g\inv f x^{-2n}\big(\prod^{2n-1}_{l=1} x^{l}fx^{-l}\ \big)hf
\\ &=(ghf)\inv h\big(\prod^{n}_{l=1}x^l([h,g\inv]gf)\inv x^{-l}\big)\big( \prod^{2n-1}_{l=n+1} x^{l}(gf)\inv x^{-l}\big) x^{2n}f\inv g\inv f x^{-2n}\big(\prod^{2n-1}_{l=1} x^{l}fx^{-l}\ \big) f\\ & \ \ \ \ \big(\prod^{n-1}_{l=1}x^{l}[gf,h]^2x^{-l}\big)x^n[gf,h][f\inv g\inv f,h][f,h]x^{-n}\big(\prod^{n-1}_{l=1}x^l [f,h]^2 x^{-l}\big)
\\ &=(ghf)\inv h\big(\prod^{n}_{l=1}x^l([h,g\inv]gf)\inv x^{-l}\big)\big( \prod^{2n-1}_{l=n+1} x^{l}(gf)\inv x^{-l}\big) x^{2n}f\inv g\inv f x^{-2n}\big(\prod^{2n-1}_{l=1} x^{l}fx^{-l}\ \big) f
\\ &= ([h,g\inv]gf)\inv x\big(([h,g\inv]gf)\inv x\big)^{n}\big((gf)\inv x\big)^{n-1} f\inv g\inv f \big(x\inv f\big)^{2n-1}x\inv f,
\end{align*}
where $[gf,h][f\inv g\inv f,h][f,h]=1$ follows by simple calculation. Also we have
\begin{align*}
&\CC^{-n-1}_{[h,g\inv]gf}\CC^{-n}_{gf}\CC^{2n+1}_f \\&=\big(([h,g\inv]gf)\inv x\big)^{n+1}\big((gf)\inv x\big)^{n-1} f\inv g\inv f \big(x\inv f\big)^{2n}
\\ &=([h,g\inv]gf)\inv x\big(([h,g\inv]gf)\inv x\big)^{n}\big((gf)\inv x\big)^{n-1} f\inv g\inv f \big(x\inv f\big)^{2n-1}x\inv f.
\end{align*}
Thus we get
\begin{equation}\label{2n1}
\CC^{-2n-1}_{ghf}\CC^{2n+1}_{hf} = \CC^{-n-1}_{[h,g\inv]gf}\CC^{-n}_{gf}\CC^{2n+1}_f.
\end{equation}
So, by comparing wreath product coordinates, we obtain
$$x^{2n+1}gx^{-2n-1}h=hx^{n+1}[h,g\inv]x^{-n-1}x^{2n+1}gx^{-2n-1}.$$

On the other hand, by commutativity we have 
\begin{align*}
hx^{2n+1}gx^{-2n-1} &= hx^{2n+1}gx^{-2n-1} x^{n+1}[h,g\inv]x^{-n-1} x^{n+1}[g,h\inv]x^{-n-1}
\\ &= h x^{n+1} (x^{n}gx^{-n} [h,g\inv]) x^{-n-1}  x^{n+1}[g,h\inv]x^{-n-1}
\\ &= h x^{n+1} [h,g\inv]x^ngx^{-n} x^{-n-1}  x^{n+1}[g,h\inv]x^{-n-1}
\\ &= (h x^{n+1} [h,g\inv]x^{-n-1} x^{2n+1}gx^{-2n-1}) x^{n+1}[g,h\inv]x^{-n-1}
\\ &= x^{2n+1}gx^{-2n-1} h x^{n+1}[g,h\inv]x^{-n-1}
\\ &= x^{2n+1}gx^{-2n-1}x^{n+1}[g,h\inv]x^{-n-1}h.
\end{align*}

For any $f,g,h\in G$, $0\leq m\leq 2n+1,$ the commutativity of $x^{2n+1}[g,h]x^{-2n-1}$ with $x^{m}fx^{-m}$ and of $x^{m}[g,h]x^{m}$ with $x^{2n+1}fx^{-2n-1}$ follows by the same trick. So we proved the case of $2n+1$.

Now we prove the case of $2n+2$.
By previous calculation, we have
$$x^{2n+2}gx^{-2n-2}h=gh\Big(\CC^{-2n-2}_{ghg_1}\CC^{2n+2}_{hg_1},\  \CC^{-2n-2}_{ghg_2}\CC^{2n+2}_{hg_2},\ \ldots,\  \CC^{-2n-2}_{ghg_k}\CC^{2n+2}_{hg_k}\Big),$$
and

$hx^{n+1}[h,g\inv]x^{-n-1}x^{2n+2}gx^{-2n-2} =$ $$gh\Big(\CC^{-n-1}_{[h,g\inv]gg_1}\CC^{-n-1}_{gg_1}\CC^{2n+2}_{g_1},\  \CC^{-n-1}_{[h,g\inv]gg_2}\CC^{-n-1}_{gg_2}\CC^{2n+2}_{g_2}, \ldots,  \CC^{-n-1}_{[h,g\inv]gg_k}\CC^{-n-1}_{gg_k}\CC^{2n+2}_{g_k}\Big).$$

Again we compare wreath product coordinates. Applying \eqref{2n1} and \eqref{inv}, for any $f\in G$,
\begin{align*}
&\CC^{-2n-2}_{ghf}\CC^{2n+2}_{hf} =\CC\inv_{ghf}\CC^{-2n-1}_{ghf}\CC^{2n+1}_{hf}\CC_{hf}
=\CC\inv_{ghf}\CC^{-n-1}_{[h,g\inv]gf}\CC^{-n}_{gf}\CC^{2n+1}_f\CC_{hf}
\\ &=(ghf)\inv x\big(([h,g\inv]gf)\inv x\big)^{n+1}\big((gf)\inv x\big)^{n-1}f\inv g\inv f \big(x\inv f\big)^{2n}x\inv hf
\\ &= (ghf)\inv \big(\prod^{n+1}_{l=1}x^l([h,g\inv]gf)\inv x^{-l}\big)\big( \prod^{2n}_{l=n+2} x^{l}(gf)\inv x^{-l}\big) x^{2n+1}f\inv g\inv f x^{-2n-1}\\&\ \ \ \ \cdot \big(\prod^{2n}_{l=1} x^{l}fx^{-l}\ \big)hf
\\ &=(ghf)\inv h\big(\prod^{n+1}_{l=1}x^l([h,g\inv]gf)\inv x^{-l}\big)\big( \prod^{2n}_{l=n+2} x^{l}(gf)\inv x^{-l}\big) x^{2n+1}f\inv g\inv f x^{-2n-1}\\&\ \ \ \ \cdot\big(\prod^{2n}_{l=1} x^{l}fx^{-l}\ \big) f \Big(\prod^n_{l=1}\big(x^l[gf,h][gf,h]x^{-l}\big)\Big)\big(x^{n+1}[f\inv g\inv f,h]x^{-n-1}\big)\\ & \ \ \ \ \cdot\Big(\prod^n_{l=1}\big(x^l[f,h][f,h]x^{-l}\big)\Big)
\\ &=(ghf)\inv h\big(\prod^{n}_{l=1}x^l([h,g\inv]gf)\inv x^{-l}\big)x^{n+1}([h,g\inv]gf)\inv[f\inv g\inv f,h]x^{-n-1}\\ & \ \ \ \ \big( \prod^{2n}_{l=n+2} x^{l}(gf)\inv x^{-l}\big) x^{2n+1}f\inv g\inv f x^{-2n-1}\big(\prod^{2n}_{l=1} x^{l}fx^{-l}\ \big) f.
\end{align*}
On the other hand,
\begin{align*}
&\CC^{-n-1}_{[h,g\inv]gf}\CC^{-n-1}_{gf}\CC^{2n+2}_f =\big(([h,g\inv]gf)\inv x\big)^{n+1}\big((gf)\inv x\big)^{n} f\inv g\inv f\big(x\inv f\big)^{2n+1}=
\\ &(ghf)\inv h\big(\prod^{n}_{l=1}x^l([h,g\inv]gf)\inv x^{-l}\big)\big( \prod^{2n}_{l=n+1} x^{l}(gf)\inv x^{-l}\big) x^{2n+1}f\inv g\inv f x^{-2n-1}\big(\prod^{2n}_{l=1} x^{l}fx^{-l}\ \big) f.
\end{align*}
Since $([h,g\inv]gf)\inv [f\inv g\inv f,h]=(gf)\inv$, then 
$\CC^{-2n-2}_{ghf}\CC^{2n+2}_{hf} =\CC^{-n-1}_{[h,g\inv]gf}\CC^{-n-1}_{gf}\CC^{2n+2}_f,$ and thus
$$x^{2n+2}gx^{-2n-2}h=hx^{n+1}[h,g\inv]x^{-n-1}x^{2n+2}gx^{-2n-2}.$$
Moreover, this implies
$$hx^{2n+2}gx^{-2n-2} = x^{2n+2}gx^{-2n-2}x^{n+1}[g,h\inv]x^{-n-1}h,$$ 
by the commutativity trick. 

For any $f,g,h\in G$, $0\leq m\leq 2n+2,$ the commutativity of $x^{2n+2}[g,h]x^{-2n-2}$ with $x^{m}fx^{-m}$ and of $x^{m}[g,h]x^{m}$ with $x^{2n+2}fx^{-2n-2}$, follows from nilpotency. So we prove the case of $2n+2$ and hence the lemma.

\end{proof}

\begin{lemma}\label{cor}
For any $g,h\in G, n,m\in \Z$, $x^n [g,h]x^{-n}$ lies in the center of $N$, and
$$x^ngx^{-n}x^mhx^{-m}=x^mhx^{-m}x^{[\frac{n+m}{2}]}[h,g\inv]x^{-[\frac{n+m}{2}]}x^ngx^{-n}.$$
\end{lemma}

\begin{proof}
We assume $n\geq m$.

Let $f\in G$. We compute
\begin{align*}
x^n[g,h]x^{-n}x^mfx^{-m} &= x^m x^{n-m}[g,h]x^{m-n}fx^{-m}
\\ &=x^m f x^{n-m}[g,h]x^{m-n}x^{-m}
\\ &= x^mfx^{-m}x^n[g,h]x^{-n},
\end{align*}
and
\begin{align*}
x^ngx^{-n}x^mhx^{-m} &=x^m x^{n-m}gx^{m-n}hx^{-m}
\\ &= x^m h x^{[\frac{n-m}{2}]}[h,g\inv]x^{-[\frac{n-m}{2}]}x^{n-m}gx^{m-n}x^{-m}
\\ &= x^m h x^{-m}x^{[\frac{n+m}{2}]}[h, g\inv]x^{-[\frac{n+m}{2}]}x^ngx^{-n}
\\ &= x^mhx^{-m}x^{[\frac{n+m}{2}]}[h,g\inv]x^{-[\frac{n+m}{2}]}x^ngx^{-n}.
\end{align*}
The case of $n\leq m$ is similar.
\end{proof}

\begin{lemma}\label{form}
Given any non-trivial torsion element $\gamma\in\GG(\CC(G))$, it can be written uniquely in the following form:
\begin{equation}\label{order}
\gamma= x^{i_1}f_1x^{-i_1}x^{i_2}f_2x^{-i_2}\dots x^{i_j}f_{j}x^{-i_j},
\end{equation}
where $i_1<i_2<\ldots <i_j$ are integers, $j$ is a positive integer, and $f_1,f_2,\ldots,f_j\in G\setminus{\{g_1=1\}}$.
\end{lemma}

\begin{proof}
Since $\gamma$ is a product of conjugates by $x$ of elements in $G$, by Lemma \ref{cor}, $\gamma$ can be written in the form of \eqref{order}. So we only need to show the uniqueness.

Let $\gamma= x^{i_1}f_1x^{-i_1}x^{i_2}f_2x^{-i_2}\dots x^{i_j}f_{j}x^{-i_j}=1$, where $ i_1<i_2<\ldots <i_j$ and $f_1,f_2,\ldots,f_j\in G$. Then $1=x^{-i_1}\gamma x^{i_1}=f_1x^{i_2-i_1}f_2x^{-i_2+i_1}\dots x^{i_j-i_1}f_{j}x^{-i_j+i_1}$ has depth $i_j-i_1+1\geq 1$, if $f_j\neq 1$. Hence, by an inductive argument, we have $f_i=1, 1\leq i\leq j$.
It completes the proof.
\end{proof}

\begin{proof}[Proof of Theorem \ref{pres} ]
Note that \cite{SiSt} $\GG(\CC(G))=N\rtimes\langle x\rangle$. Then the theorem follows from Lemma \ref{cor} and Lemma \ref{form}.
\end{proof}

In the end of \cite{KSS}, Kambites, Silva and Steinberg conjectured that the automata group associated to any finite non-abelian group can not have bounded torsion. The following corollary disproves the conjecture.

\begin{corollary}
Let $G$ be a finite nilpotent group of class 2. Assume the square of each element of $G$ is central, then $\GG(\CC(G))$ has bounded torsion.
\end{corollary}
\begin{proof}
Let $|G|=k$. Note that $k$ is even. Let $\ga$ be a torsion element. Since every commutator has order at most 2, then by Lemma \ref{form} and Theorem \ref{pres} we have $\ga^{2k}=1$, i.e., the order of any torsion element is a factor of $2k$.
\end{proof}

\subsection{Non-embedding into wreath products}

\begin{lemma}
Let $G$ be a finite 2-group of nilpotency class 2. Then $G$ has a subgroup $H$ in which each square element is central in $H$.
\end{lemma}

\begin{proof}
By assumption on $G$, there exists a non-central element $g\in G$ such that  $g^2$ is central. Similarly, there exists $h\in G$ so that it does not commute with $g$ but $h^2$ commutes with $g$. Then the subgroup $H$ of $G$ generated by $g$ and $h$ is nilpotent of class 2. We show each square element in $H$ is central in $H$. Since $g^2$ and $h^2$ are central, it suffices to show for  squares of alternate products of $g$ or $g\inv$ with $h$ or $h\inv$. Since $g\inv = g g\inv g\inv$ and $h\inv=hh\inv h\inv$, it reduces to squares of alternate products of $g$ with $h$. Since $(ghg)^2=ghgghg=g^4h^2$, by symmetry we only need to show that $(gh)^2$ is central. Indeed, we have $ghghh(hghgh)\inv=ghghhh\inv g\inv h\inv g\inv h\inv=gh(gh g\inv h\inv) g\inv h\inv=ghg\inv gh g\inv h\inv h\inv=ghhg\inv h\inv h\inv=1$, and similarly $ghghg(gghgh)\inv=1$ by symmetry.

\end{proof}

\begin{proof}[Proof of Theorem \ref{conje}]
Let $G'$ be the 2-group. Let $g$ be a non-central element in $G'$ such that  $g^2$ is central. Let $h\in G'$ so that it does not commute with $g$ but $h^2$ commutes with $g$. Then the subgroup $H$ of $G$ generated by $g$ and $h$ is nilpotent of class 2. By previous lemma, each square element of $H$ is central in $H$. 

For each $n\in\N$, we consider the element $$\gamma_n=(x^n hx^{-n})g(x^nh\inv x^{-n}).$$
Each such element is a conjugate of the torsion element $g$ by another torsion element $x^nhx^{-n}$. $\gamma_n$ has depth at most $n+1$. Since every conjugacy class in the torsion subgroup of any wreath product of a finite group with a torsion-free group is finite, see \cite{KSS} Lemma 6.4, it suffices to show that $\gamma_n$ has depth at least $[\frac{n}{2}]$. Indeed, $\ga_n=gx^{[\frac{n}{2}]}[g,h\inv]x^{-[\frac{n}{2}]}$ has depth $[\frac{n}{2}]+1$ by Lemma \ref{depth}.

\end{proof}

\begin{proof}[Proof of Theorem \ref{embed}]
Combining Theorem \ref{conje} with results in the appendix of \cite{KSS}, the theorem follows.
\end{proof}

\section{Cross-wired Lamplighter groups}

The Cayley graph of lamplighter groups can be examples of Diestel-Leader graphs, for particular choices of generators \cite{DL, Moeller, Woess, Wortman}. {\it Cross-wired lamplighter groups} \cite{CFK} are defined to be cocompact lattices of the isometry group of Diestel-Leader graphs. 
Cornulier, Fisher and Kashyap \cite{CFK} give a necessary and sufficient condition for a locally compact group to be isomorphic to a closed cocompact subgroup in the isometry group of a Diestel-Leader graph. The following is the sufficient condition for the case of discrete groups.

\begin{theorem}\cite{CFK}
Let $\Gamma$ be a discrete group, with a semidirect product decomposition $\Gamma=H\rtimes\langle t\rangle$, where $H$ is infinite and $t$ is torsion free. Assume that $H$ has subgroups $L$ and $L'$ such that
\begin{itemize}
\item $tLt^{-1}$ and $t^{-1}L't$ are finite index subgroups, of index $m$ and $n$, in $L$ and $L'$ respectively;
\item $\bigcup_{k\in\Z}t^{-k}Lt^k=\bigcup_{k\in\Z}t^{k}L't^{-k}=H$ (these are increasing unions);
\item $L\cap L'$ is finite;
\item the double coset space $L\backslash H/L'$ is finite.
\end{itemize}
Then $\Ga$ is finitely generated, $m=n$ and $\Gamma$ has a proper, transitive action on some Diestel-Leader graph.
\end{theorem}

\subsection{Proof of Theorem \ref{lamp}}
\begin{proof}[Proof of Theorem \ref{lamp}]
To show the automata groups are cross-wired lamplighters, 
it suffices to check the conditions in above theorem.

Let $G=\{g_1=1,g_2,...,g_k\}, L=\langle x^nGx^{-n},n\geq 0\rangle$, and  $L'=\langle x^nGx^{-n},n\leq -1\rangle$.

(1) We show that $xLx\inv $ and $x\inv L'x$ are finite index subgroups in $L$ and $L'$ respectively. 
Note that $xLx\inv =\langle x^nGx^{-n},n\geq 1\rangle< L$. Lemma \ref{keylemma} implies that every element in $L$ can be written as an element in $g(xLx\inv)$ for some $g\in G$. So the index of $xLx\inv$ in $L$ is finite. Similarly $x\inv L'x=\langle x^nGx^{-n},n\leq -1\rangle$ has finite index in $L'$.

(2) We show that $\bigcup_{n\in \Z}x^{-n}Lx^n=\bigcup_{n\in\Z}x^nL'x^{-n}=N,$ and the unions are increasing. Note that $\bigcup_{n\in\Z} x^{-n}Lx^n$ is a subgroup of $N$ and $\bigcup_{n\in\Z} x^{-n}Lx^n$ contains all the generators of $N$. Thus  $\bigcup_{n\in \Z}x^{-n}Lx^n=N.$ Similarly we have $\bigcup_{n\in\Z}x^nL'x^{-n}=N.$ The unions are the unions of increasing subgroups, by the definitions of $L$ and $L'$.

(3) We show $L\cap L'$ is finite. Indeed, it follows from Lemma \ref{form} that $L\cap L'$ is trivial. We give an alternative proof below using the idea of depth.
By Lemma \ref{depth}, each generator of $L$ has finite depth, so do the elements of $L$.
It suffices to show every non-trivial element of $L'$ has infinite depth. Let $1\neq g\in G$. Since 
\begin{align*}
x\inv gx &= (\CC_{g_1},\CC_{g_2},\ldots,\CC_{g_k})g(\CC\inv _{g_1},\CC\inv_{g_2},\ldots,\CC\inv_{g_k})
\\ &= g\big(\CC_{gg_1}\CC\inv_{g_1},\CC_{gg_2}\CC\inv_{g_2},\ldots,\CC_{gg_k}\CC\inv_{g_k}\big)
\\ &= g\big(x\inv gg_1(g_1)\inv x, x\inv gg_2(g_2)\inv x,\ldots,x\inv gg_k(g_k)\inv x\big)
\\ &= g(x\inv gx,x\inv gx,\ldots,x\inv gx), 
\end{align*}
then $x\inv gx$ acts on any infinite word in $G^\omega$ by multiplying $g$ on the left to each letter in the word, and hence has infinite depth. Assume $x^ngx^{-n}$ has finite depth $m$ for some $n\leq -2$. Then by lemma \ref{depth}, $x\inv gx$ has depth at most $m-n-1$, which is a contradiction. Thus $x^ngx^{-n}$ has infinite depth for each $n\leq -1$. 

Let $$1\neq\gamma=x^{n_1}f_1x^{-n_1}x^{n_2}f_2x^{-n_2}\ldots x^{n_p}f_px^{-n_p}$$ be a reduced element of $L'$, where $p\in \N, n_1, n_2,\ldots,n_p\leq -1$. Then by Lemma \ref{form}, we can rewrite this element as
$$x^{m_1}h_1x^{-m_1}x^{m_2}h_2x^{-m_2}\ldots x^{m_q}h_qx^{-m_q}, $$ for some $q\in\N, m_1<m_2<\ldots<m_q\leq -1$. Then 
\begin{align*}
& \ \ \ \ x^{-m_1-1}\big(x^{m_1}h_1x^{-m_1}x^{m_2}h_2x^{-m_2}\ldots x^{m_q}h_qx^{-m_q}\big)x^{m_1+1}  \\ &=x^{-1}h_1x x^{m_2-m_1-1}h_2x^{-m_2+m_1+1}\ldots x^{m_{q}-m_1-1}h_qx^{-m_{q}+m_1+1},
\end{align*}
has infinite depth, since all conjugates of $h_i, i\geq 2$ in the product have finite depth but $x\inv h_1x$. Thus, by lemma \ref{depth}, $\gamma$ has infinite depth. Therefore, $L\cap L'$ is trivial.

(4) We show that the double coset space $L\backslash N/L'$ is finite. It follows from Lemma \ref{cor}, as shown in the proof of (3) that products of conjugates can be written as products of conjugates in either increasing or decreasing order, that $L\backslash N/L'$ is trivial.
\end{proof}

\noindent{\bf Remark:}
A cross-wired lamplighter is said to be $symmetric$ \cite{CFK} if it admits an automorphism $\alpha$ such that if $\pi$ is the projection to $\Z$ (which is unique up to sign) then $\pi \circ \alpha =-\pi$, and virtually symmetric if it admits a symmetric subgroup of finite index. $\GG(\CC(D_8))$ is symmetric, but $\GG(\CC(Q_8))$ is not  symmetric.

\noindent{\bf Remark:}
The lamplighter group $F \wr\Z$ can be interpreted as follows. The generator of $\Z$ describes the walk of a lamplighter along an infinite street with a streetlamp at each integer. Each streetlamp has $|F|$ states. The generators of $F$ represent the ability of the lamplighter to change the state of the lamp at his current position. This description is most intuitive where $F = \Z_2$ and lamps are simply on or off. For the group $G$ in Theorem \ref{pres}, the cross-wired lamplighter $\GG(\CC(G))=\langle G, x\rangle$ can be interpreted in a similar but more complicated way. The generator $x$ still describe the walk of a lamplighter. Each streetlamp has $|G|$ states, including off as a state represented by $1\in G$. The states of lamps are determined by the normal form in Lemma \ref{form}. The generators of $G$ now represent the ability of the lamplighter to change the states of the lamps not only at his current position, but also at some positions between his and the [mid] position of his and the position of the lamp to the farthest right (positive direction) which is on. The change of state of lamps on the right is up to a commutator state. This interpretation gives one cross-wired pattern for cross-wired lamplighter groups.

\section{Linearity of automata groups}

\subsection{Notations}

 Let $\ut_m(\Z_2[t^{\pm 1}])$ be $m\times m$ upper unitriangular matrices over Laurent polynomial ring $\Z_2[t^{\pm 1}]$. Let $g=(g_{ij}), h=(h_{ij})$ be elements in $\ut_m(\Z_2[t^{\pm 1}])$ and $\bar{g}=(\bar{g}_{ij}), \bar{h}=(\bar{h}_{ij})$ be their inverses.
\begin{lemma}
$\bar{g}_{ii}=1,\bar{g}_{i,i+1}=g_{i,i+1}$, and $\bar{g}_{ij},\  j>i,$ satisfies the recursive formular $\bar{g}_{ij}= \sum\limits_{k=i+1}^{j}g_{ik}\bar{g}_{kj}$.
\end{lemma}
\begin{proof}
$\bar{g}_{ii}=1$. For $j>i$, we have
$0=\sum\limits_{l=i}^{j}g_{il}\bar{g}_{lj} =g_{ii}\bar{g}_{ij}+\sum\limits_{l=i+1}^{j}g_{il}\bar{g}_{lj}=\bar{g}_{ij}+\sum\limits_{l=i+1}^{j}g_{il}\bar{g}_{lj}.$
\end{proof}

\begin{lemma}\label{gh1m}
Suppose $[g,h]_{i,j}=0$ as long as $0<j-i<m-1$ , then $$[g,h]_{1m}=\sum\limits_{k=2}^{m-1}(g_{1k}h_{km}+h_{1k}g_{km}).$$
\end{lemma}

\begin{proof}
\begin{align*}
[g,h]_{1m} &=\sum\limits_{i=1}^m \sum\limits_{k=i}^m\sum\limits_{j=k}^m\bar{g}_{1i}\bar{h}_{ik}g_{kj}h_{jm}
\\ &=\sum\limits_{k=1}^m\sum\limits_{j=k}^m\bar{h}_{1k}g_{kj}h_{jm}+\sum\limits_{i=2}^m \sum\limits_{k=i}^m\sum\limits_{j=k}^m\bar{g}_{1i}\bar{h}_{ik}g_{kj}h_{jm}
\\ &= \sum\limits_{j=1}^mg_{1j}h_{jm}+\sum\limits_{k=2}^m\sum\limits_{j=k}^m\sum\limits_{l=2}^kh_{1l}\bar{h}_{lk}g_{kj}h_{jm}+\sum\limits_{i=2}^m \sum\limits_{k=i}^m\sum\limits_{j=k}^m\sum\limits_{l=2}^ig_{1l}\bar{g}_{li}\bar{h}_{ik}g_{kj}h_{jm}
\\ &= \sum\limits_{j=1}^mg_{1j}h_{jm}+\sum\limits_{l=2}^m\sum\limits_{k=l}^m\sum\limits_{j=k}^mh_{1l}\bar{h}_{lk}g_{kj}h_{jm}+\sum\limits_{l=2}^m \sum\limits_{i=l}^m\sum\limits_{k=i}^m\sum\limits_{j=k}^mg_{1l}\bar{g}_{li}\bar{h}_{ik}g_{kj}h_{jm}
\\ &=\sum\limits_{j=1}^mg_{1j}h_{jm}+\sum\limits_{l=2}^mh_{1l}\sum\limits_{k=l}^m\sum\limits_{j=k}^m\bar{h}_{lk}g_{kj}h_{jm}+\sum\limits_{l=2}^m g_{1l}\sum\limits_{i=l}^m\sum\limits_{k=i}^m\sum\limits_{j=k}^m\bar{g}_{li}\bar{h}_{ik}g_{kj}h_{jm}
\\ &=\sum\limits_{j=1}^mg_{1j}h_{jm} + \sum\limits_{l=2}^mh_{1l}g_{lm}+\sum\limits_{l=2}^{m-1} g_{1l}0 +g_{1m}
\\ &=\sum\limits_{k=2}^{m-1}(g_{1k}h_{km}+h_{1k}g_{km}).
\end{align*}
\end{proof}

We recall here the commutator identities in group theory: $$[a, bc]=[a,c]c\inv[a,b]c,\ [ab,c]=b\inv[a,c]b[b,c],$$ where $a, b, c$ are elements of a group $F$.

For the group $G$ in Theorem \ref{pres}, $\GG(\CC(G))=\langle G, x\rangle$. Let $\GG_2(\CC(G))=\langle G, xGx\inv, x^2\rangle$. Then it is an index two subgroup of $\GG(\CC(G))$.

\subsection{Linearity of $\GG(\CC(Q_8))$}

Let $Q_8=\langle a, b \ |\ a^4=1, a^2=b^2, b\inv ab=a\inv\rangle.$ It has center $Z(Q_8)=\langle a^2\rangle$.
Then $\GG(\CC(Q_8))$ is generated by $\{a, b, x\}$, and $\GG_2(\CC(Q_8))$ is generated by $\{a, b, xax\inv, xbx\inv, x^2\}$.

\begin{lemma}\label{q8form}
Any non-trivial torsion element of $\GG(\CC(Q_8))$ can be uniquely written in the following  form:
$$\big(\prod_{r=1}^{k}x^{i_r}ax^{-i_r}\big)\big(\prod_{r=1}^{l}x^{j_r}bx^{-j_r}\big)\big(\prod_{r=1}^{n}x^{m_r}a^2x^{-m_r}\big),$$
 where $i_1<i_2<\cdots<i_k, j_1<j_2<\cdots<j_l,$ and $m_1<m_2<\cdots <m_n$. 
\end{lemma}

\begin{proof}
It follows from Lemma \ref{form} and the group structure of $Q_8$.
\end{proof}

\begin{theorem}
$\GG_2(\CC(Q_8))$ embeds into $\gl_6 (\Z_2[t^{\pm 1}])$.
\end{theorem}

\begin{proof}
Let $\alpha$ be the map from $\{a, b, xax\inv, xbx\inv, x^2\}$ to $\gl_6 (\Z_2[t^{\pm 1}])$ satisfying

$\alpha(a)=\left( \begin{array}{cccccc}
 1 & t & 0 & 0 & 0 & 0\\
  0 & 1 & 0 & 0 & 0 & 1\\
  0 & 0 & 1 & 0 & 0 & 0\\
  0 & 0 & 0 & 1 & 0 & 0\\
  0 & 0 & 0 & 0 & 1 & 0\\ 
  0 & 0 & 0 & 0 & 0 & 1
  \end{array}\right),\ \ 
  \alpha(b)=\left( \begin{array}{cccccc}
 1 & 0 & t & 0 & 0 & 0\\
  0 & 1 & 0 & 0 & 0 & 1\\
  0 & 0 & 1 & 0 & 0 & 1\\
  0 & 0 & 0 & 1 & 0 & 0\\
  0 & 0 & 0 & 0 & 1 & 0\\ 
  0 & 0 & 0 & 0 & 0 & 1
  \end{array}\right), \\
\alpha(xax\inv)=\left( \begin{array}{cccccc}
 1 & 1 & 0 & 1 & 0 & 0\\
  0 & 1 & 0 & 0 & 0 & t\inv\\
  0 & 0 & 1 & 0 & 0 & 0\\
  0 & 0 & 0 & 1 & 0 & t\inv+1\\
  0 & 0 & 0 & 0 & 1 & 1\\ 
  0 & 0 & 0 & 0 & 0 & 1
  \end{array}\right),\ \ 
\alpha(xbx\inv)=\left( \begin{array}{cccccc}
 1 & 0 & 1 & 0 & 1 & 0\\
  0 & 1 & 0 & 0 & 0 & t\inv\\
  0 & 0 & 1 & 0 & 0 & t\inv\\
  0 & 0 & 0 & 1 & 0 & t\inv\\
  0 & 0 & 0 & 0 & 1 & t\inv+1\\ 
  0 & 0 & 0 & 0 & 0 & 1
  \end{array}\right), \\
\alpha(x^2)=\left( \begin{array}{cccccc}
 1 & 0 & 0 & 0 & 0 & 0\\
  0 & t & 0 & 0 & 0 & 0\\
  0 & 0 & t & 0 & 0 & 0\\
  0 & 0 & 0 & t & 0 & 0\\
  0 & 0 & 0 & 0 & t & 0\\ 
  0 & 0 & 0 & 0 & 0 & t^2
  \end{array}\right).
 $

Easily $\{\alpha(a),\alpha(b)\}$ and $\{\alpha(xax\inv),\alpha(xbx\inv)\}$ generates copies of $Q_8$. Next, we show that the map $\alpha$ extends to a group homomorphism from $\GG_2(\CC(Q_8))$ to $\gl_6 (\Z_2[t^{\pm 1}])$.

Extending $\alpha$ to generators of the torsion subgroup by conjugate of $\alpha(x^2)$, we have
\\ \\$\alpha(x^{2n}ax^{-2n})=\left( \begin{array}{cccccc}
 1 & t^{-n+1} & 0 & 0 & 0 & 0\\
  0 & 1 & 0 & 0 & 0 & t^{-n}\\
  0 & 0 & 1 & 0 & 0 & 0\\
  0 & 0 & 0 & 1 & 0 & 0\\
  0 & 0 & 0 & 0 & 1 & 0\\ 
  0 & 0 & 0 & 0 & 0 & 1
  \end{array}\right),
  \alpha(x^{2n}bx^{-2n})=\left( \begin{array}{cccccc}
 1 & 0 & t^{-n+1} & 0 & 0 & 0\\
  0 & 1 & 0 & 0 & 0 & t^{-n}\\
  0 & 0 & 1 & 0 & 0 & t^{-n}\\
  0 & 0 & 0 & 1 & 0 & 0\\
  0 & 0 & 0 & 0 & 1 & 0\\ 
  0 & 0 & 0 & 0 & 0 & 1
  \end{array}\right),\\
\alpha(x^{2n+1}ax^{-2n-1})=\left( \begin{array}{cccccc}
 1 & t^{-n} & 0 & t^{-n} & 0 & 0\\
  0 & 1 & 0 & 0 & 0 & t^{-n-1}\\
  0 & 0 & 1 & 0 & 0 & 0\\
  0 & 0 & 0 & 1 & 0 & t^{-n-1}+t^{-n}\\
  0 & 0 & 0 & 0 & 1 & t^{-n}\\ 
  0 & 0 & 0 & 0 & 0 & 1
  \end{array}\right), \\
\alpha(x^{2n+1}bx^{-2n-1})=\left( \begin{array}{cccccc}
 1 & 0 & t^{-n} & 0 & t^{-n} & 0\\
  0 & 1 & 0 & 0 & 0 & t^{-n-1}\\
  0 & 0 & 1 & 0 & 0 & t^{-n-1}\\
  0 & 0 & 0 & 1 & 0 & t^{-n-1}\\
  0 & 0 & 0 & 0 & 1 & t^{-n-1}+t^{-n}\\ 
  0 & 0 & 0 & 0 & 0 & 1
  \end{array}\right),\\
   \alpha(x^{n}a^2x^{-n})=\left( \begin{array}{cccccc}
 1 & 0 & 0 & 0 & 0 & t^{-n+1}\\
  0 & 1 & 0 & 0 & 0 & 0\\
  0 & 0 & 1 & 0 & 0 & 0\\
  0 & 0 & 0 & 1 & 0 & 0\\
  0 & 0 & 0 & 0 & 1 & 0\\ 
  0 & 0 & 0 & 0 & 0 & 1
  \end{array}\right).$

To show $\alpha$ extends to a group homomorphism, it suffices to check those relations in Theorem \ref {pres}. Indeed, applying Lemma \ref{gh1m}, $\alpha(x^nax^{-n})$'s commute with each other, so do $\alpha(x^nbx^{-n})$'s. Moreover,  we calculate
\begin{align*}
& [\alpha(x^{2n}ax^{-2n}), \alpha(x^{2m}bx^{-2m})]\\ &=\left[\left( \begin{array}{cccccc}
 1 & t^{-n+1} & 0 & 0 & 0 & 0\\
  0 & 1 & 0 & 0 & 0 & t^{-n}\\
  0 & 0 & 1 & 0 & 0 & 0\\
  0 & 0 & 0 & 1 & 0 & 0\\
  0 & 0 & 0 & 0 & 1 & 0\\ 
  0 & 0 & 0 & 0 & 0 & 1
  \end{array}\right), \left( \begin{array}{cccccc}
 1 & 0 & t^{-m+1} & 0 & 0 & 0\\
  0 & 1 & 0 & 0 & 0 & t^{-m}\\
  0 & 0 & 1 & 0 & 0 & t^{-m}\\
  0 & 0 & 0 & 1 & 0 & 0\\
  0 & 0 & 0 & 0 & 1 & 0\\ 
  0 & 0 & 0 & 0 & 0 & 1
  \end{array}\right)\right]\\ &= \left(\begin{array}{cccccc}
 1 & 0 & 0 & 0 & 0 & t^{-m-n+1}\\
  0 & 1 & 0 & 0 & 0 & 0\\
  0 & 0 & 1 & 0 & 0 & 0\\
  0 & 0 & 0 & 1 & 0 & 0\\
  0 & 0 & 0 & 0 & 1 & 0\\ 
  0 & 0 & 0 & 0 & 0 & 1
  \end{array}\right) =\alpha([x^{2n}ax^{-2n},x^{2m}bx^{-2m}]),
\end{align*}

\begin{align*}
& [\alpha(x^{2n}ax^{-2n}), \alpha(x^{2m+1}bx^{-2m-1})]\\ & =\left[\left( \begin{array}{cccccc}
 1 & t^{-n+1} & 0 & 0 & 0 & 0\\
  0 & 1 & 0 & 0 & 0 & t^{-n}\\
  0 & 0 & 1 & 0 & 0 & 0\\
  0 & 0 & 0 & 1 & 0 & 0\\
  0 & 0 & 0 & 0 & 1 & 0\\ 
  0 & 0 & 0 & 0 & 0 & 1
  \end{array}\right), \left( \begin{array}{cccccc}
 1 & 0 & t^{-m} & 0 & t^{-m} & 0\\
  0 & 1 & 0 & 0 & 0 & t^{-m-1}\\
  0 & 0 & 1 & 0 & 0 & t^{-m-1}\\
  0 & 0 & 0 & 1 & 0 & t^{-m-1}\\
  0 & 0 & 0 & 0 & 1 & t^{-m-1}+t^{-m}\\ 
  0 & 0 & 0 & 0 & 0 & 1
  \end{array}\right)\right] \\ &= \left(\begin{array}{cccccc}
 1 & 0 & 0 & 0 & 0 & t^{-n-m}\\
  0 & 1 & 0 & 0 & 0 & 0\\
  0 & 0 & 1 & 0 & 0 & 0\\
  0 & 0 & 0 & 1 & 0 & 0\\
  0 & 0 & 0 & 0 & 1 & 0\\ 
  0 & 0 & 0 & 0 & 0 & 1
  \end{array}\right)=\alpha([x^{2n}ax^{-2n},x^{2m+1}bx^{-2m-1}]),
\end{align*}

\begin{align*}
& [\alpha(x^{2n+1}ax^{-2n-1}), \alpha(x^{2m}bx^{-2m})]\\ &=\left[\left( \begin{array}{cccccc}
 1 & t^{-n} & 0 & t^{-n} & 0 & 0\\
  0 & 1 & 0 & 0 & 0 & t^{-n-1}\\
  0 & 0 & 1 & 0 & 0 & 0\\
  0 & 0 & 0 & 1 & 0 & t^{-n-1}+t^{-n}\\
  0 & 0 & 0 & 0 & 1 & t^{-n}\\ 
  0 & 0 & 0 & 0 & 0 & 1
  \end{array}\right), \left( \begin{array}{cccccc}
 1 & 0 & t^{-m+1} & 0 & 0 & 0\\
  0 & 1 & 0 & 0 & 0 & t^{-m}\\
  0 & 0 & 1 & 0 & 0 & t^{-m}\\
  0 & 0 & 0 & 1 & 0 & 0\\
  0 & 0 & 0 & 0 & 1 & 0\\ 
  0 & 0 & 0 & 0 & 0 & 1
  \end{array}\right)\right] \\ &= \left(\begin{array}{cccccc}
 1 & 0 & 0 & 0 & 0 & t^{-n-m}\\
  0 & 1 & 0 & 0 & 0 & 0\\
  0 & 0 & 1 & 0 & 0 & 0\\
  0 & 0 & 0 & 1 & 0 & 0\\
  0 & 0 & 0 & 0 & 1 & 0\\ 
  0 & 0 & 0 & 0 & 0 & 1
  \end{array}\right)=\alpha([x^{2n+1}ax^{-2n-1},x^{2m}bx^{-2m}]),
\end{align*}

\begin{align*}
& [\alpha(x^{2n+1}ax^{-2n-1}), \alpha(x^{2m+1}bx^{-2m-1})]\\ &=\left[\left( \begin{array}{cccccc}
 1 & t^{-n} & 0 & t^{-n} & 0 & 0\\
  0 & 1 & 0 & 0 & 0 & t^{-n-1}\\
  0 & 0 & 1 & 0 & 0 & 0\\
  0 & 0 & 0 & 1 & 0 & t^{-n-1}+t^{-n}\\
  0 & 0 & 0 & 0 & 1 & t^{-n}\\ 
  0 & 0 & 0 & 0 & 0 & 1
  \end{array}\right), \left( \begin{array}{cccccc}
 1 & 0 & t^{-m} & 0 & t^{-m} & 0\\
  0 & 1 & 0 & 0 & 0 & t^{-m-1}\\
  0 & 0 & 1 & 0 & 0 & t^{-m-1}\\
  0 & 0 & 0 & 1 & 0 & t^{-m-1}\\
  0 & 0 & 0 & 0 & 1 & t^{-m-1}+t^{-m}\\ 
  0 & 0 & 0 & 0 & 0 & 1
  \end{array}\right)\right]\\ &= \left(\begin{array}{cccccc}
 1 & 0 & 0 & 0 & 0 & t^{-n-m}\\
  0 & 1 & 0 & 0 & 0 & 0\\
  0 & 0 & 1 & 0 & 0 & 0\\
  0 & 0 & 0 & 1 & 0 & 0\\
  0 & 0 & 0 & 0 & 1 & 0\\ 
  0 & 0 & 0 & 0 & 0 & 1
    \end{array}\right)=\alpha([x^{2n+1}ax^{-2n-1},x^{2m+1}bx^{-2m-1}]).
\end{align*}
Therefore, the compatibility of $\alpha$ with the relations in Theorem \ref{pres} follows from the commutator identities in group theory.

Finally, we show the homomorphism $\alpha$ has trivial kernel. First note that any torsion free element does not lie in the kernel. Let $\gamma$ be a torsion element in $\GG(\CC(Q_8))$. It has the normal form as in Lemma \ref{q8form}. Let $\alpha(\gamma)=\big(\alpha(\gamma)_{ij}\big)  =I,$ the identity matrix.   $\alpha(\gamma)_{14}=0$ implies that there are no odd conjugates of $a$ in the normal form of $\gamma$. Then,  $\alpha(\gamma)_{12}=0$ implies that there are no conjugates of $a$ in $\gamma$. Moreover, $\alpha(\gamma)_{15}=0$ implies that there are no odd conjugates of $b$ in $\gamma$ and hence $\alpha(\gamma)_{13}=0$ implies that there are no conjugates of $b$ in $\gamma$. $\alpha(\gamma)_{16}=0$ implies that there are no conjugates of $a^2$ in $\gamma$. So, $\alpha(\gamma)=I$ implies that $\gamma$ is trivial.
\end{proof}

\noindent{\bf Remark:}
Since virtually linear implies linear, then we have the linearity of $\GG(\CC(Q_8))$.

\subsection{Linearity of $\GG(\CC(M_{2^n}))$}

Now, we consider the Iwasawa/modular group $M_{2^n}$ of order $2^n$, $n\geq 3$, where $M_{2^n}=\langle a, b \ |\ a^{2^{n-1}}=b^2=e,  bab\inv=a^{2^{n-2}+1} \rangle.$ It has center $Z(M_{2^n})=\langle a^2\rangle$ and derived subgroup $[M_{2^n},M_{2^n}]=\langle a^{2^{n-2}}\rangle$. Actually, $G=\{1, a, a^2,...,a^{2^{n-1}-1}, b, ab, a^2b,..., a^{2^{n-1}-1}b\}$. 

$\GG(\CC(M_{2^n}))=\langle a, b, x\rangle$. $\GG_2(\CC(M_{2^n}))=\langle a, b, xax\inv, xbx\inv, x^2\rangle$ is a index two subgroup of $\GG(\CC(M_{2^n}))$. 
\begin{lemma}\label{m2nform}
Any non-trivial torsion element of $\GG(\CC(M_{2^n}))$ can be uniquely written in the following form:
$$\prod_{r=1}^{k}x^{i_r}a^{s_r}x^{-i_r} \prod_{r=1}^{l}x^{j_r}bx^{-j_r} \prod_{r=1}^{n}x^{m_r}[a,b]x^{-m_r},$$
where $i_1<i_2<\cdots<i_k, j_1<j_2<\cdots<j_l,$  $m_1<m_2<\cdots <m_n$, and $1\leq s_1, s_2,\cdots, s_k\leq 2^{n-2}-1.$
\end{lemma}

\begin{proof}
It follows from Lemma \ref{form} and the group structure of $M_{2^n}$.
\end{proof}

\begin{theorem}
$\GG_2(\CC(M_{2^n}))$ embeds into $\gl_{2^{n-1}+2} (\Z_2[t^{\pm 1}])$.
\end{theorem}

\begin{proof}
Let $E_{i,j}$, $1\leq i,j\leq 2^{n-1}+2$, denote the $(2^{n-1}+2)\times(2^{n-1}+2)$ matrix with 1 in the $i$-th row $j$-th column and 0 everywhere else.  $E_{i,i}$ is simply denoted by $E_{i}$.

Let $\alpha$ be the map from the generating set $\{a, b, xax\inv, xbx\inv, x^2\}$ to $\gl_{2^{n-1}+2} (\Z_2[t^{\pm 1}])$ satisfying:

$$\alpha(a)=I+E_{1,2}+\sum\limits_{m=1}^{2^{n-2}-2}E_{2m,2m+2}+E_{2^{n-1}-2,2^{n-1}+2}+E_{1,2^{n-1}},$$
$$\alpha(b)=I+E_{2^{n-1},2^{n-1}+2}+E_{2^{n-1}+1,2^{n-1}+2},$$
$$\alpha(xax\inv)=I+t^{-1}\sum\limits_{m=1}^{2^{n-3}}E_{2m-1,2m+1}+\sum\limits_{m=2^{n-3}+1}^{2^{n-2}-1}E_{2m-1,2m+1}+E_{2^{n-1}-1,2^{n-1}+2}+t^{-2^{n-3}}E_{1,2^{n-1}+1},$$
$$\alpha(xbx\inv)=I+t^{-2^{n-3}}E_{2^{n-1},2^{n-1}+2}+E_{2^{n-1}+1,2^{n-1}+2},$$
$$\alpha(x^2)=E_{1}+\sum\limits_{m=1}^{2^{n-2}-1}(t^{m}E_{2m}+t^{m}E_{2m+1})+t^{2^{n-3}}E_{2^{n-1}}+t^{2^{n-3}}E_{2^{n-1}+1}+t^{2^{n-2}}E_{2^{n-1}+2}.$$

It is easy to verify that $\alpha(a)$ has order $2^{n-1}$ and $\alpha(a)^{2^{n-2}}=I+E_{1,2^{n-1}+2}$. $\alpha(a)$ has $2^{n-2}+1$ non-zero entries above the diagonal, and $\alpha(a)^m$, $2\leq m\leq 2^{n-1}$, has at least $2^{n-1}-m$ non-zero entries above the diagonal. Indeed, 
$$\alpha(a)^2=I+E_{1,4}+\sum\limits_{m=1}^{2^{n-2}-3}E_{2m,2m+4}+E_{2^{n-1}-4,2^{n-1}+2},$$
$$\alpha(a)^4=I+E_{1,8}+\sum\limits_{m=1}^{2^{n-2}-5}E_{2m,2m+8}+E_{2^{n-1}-8,2^{n-1}+2},$$
$$\cdots\cdots$$
$$\alpha(a)^{2^j}=I+E_{1,2^{j+1}}+\sum\limits_{m=1}^{2^{n-2}-2^j-1}E_{2m,2m+2^{j+1}}+E_{2^{n-1}-2^{j+1},2^{n-1}+2},$$
$$\cdots\cdots$$
$$\alpha(a)^{2^{n-3}}=I+E_{1,2^{n-2}}+\sum\limits_{m=1}^{2^{n-2}-2^{n-3}-1}E_{2m,2m+2^{n-2}}+E_{2^{n-1}-2^{n-2},2^{n-1}+2},$$
$$\alpha(a)^{2^{n-2}}=I+E_{1,2^{n-1}+2}, $$
$$ \alpha(a)^{2^{n-1}}=I.$$
Moreover, we have similar properties with $\alpha(xax^{-1})$ because of similar construction:
$$\alpha(xax\inv)^{2^{n-2}}=I+t^{-2^{n-3}}E_{1,2^{n-1}+2}, $$
$$ \alpha(xax\inv)^{2^{n-1}}=I.$$

Applying Lemma \ref{gh1m} where the assumptions are easily verified, we see that 
$$[\alpha(a),\alpha(b)]=I+E_{1,2^{n-1}+2}=\alpha(a)^{2^{n-2}}, $$
$$[\alpha(xax^{-1}),\alpha(xbx^{-1})]=I+t^{-2^{n-3}}E_{1,2^{n-1}+2}=\alpha(xax^{-1})^{2^{n-2}}.$$
Since $\alpha(b)^2=\alpha(xbx\inv)^2=I$, then $\{\alpha(a), \alpha(b)\}$ and $\{\alpha(xax\inv), \alpha(xbx\inv)\}$ generates two copies of $M_{2^n}$.

Next, we show that $\alpha$ extends to a group homomorphism. Extending $\alpha$, we have 
$$\alpha(x^{2k}ax^{-2k})=I+t^{-k}E_{1,2}+t^{-k}\sum\limits_{m=1}^{2^{n-2}-2}E_{2m,2m+2}+t^{-k}E_{2^{n-1}-2,2^{n-1}+2}+t^{-2^{n-3}k}E_{1,2^{n-1}},$$
$$\alpha(x^{2k}bx^{-2k})=I+t^{-2^{n-3}k}E_{2^{n-1},2^{n-1}+2}+t^{-2^{n-3}k}E_{2^{n-1}+1,2^{n-1}+2},$$
$$\alpha(x^{2k+1}ax^{-2k-1})=I+t^{-1-k}\sum\limits_{m=1}^{2^{n-3}}E_{2m-1,2m+1}+t^{-k}\sum\limits_{m=2^{n-3}+1}^{2^{n-2}-1}E_{2m-1,2m+1}$$ $$+t^{-k}E_{2^{n-1}-1,2^{n-1}+2}+t^{-2^{n-3}(k+1)}E_{1,2^{n-1}+1},$$
$$\alpha(x^{2k+1}bx^{-2k-1})=I+t^{-2^{n-3}(k+1)}E_{2^{n-1},2^{n-1}+2}+t^{-2^{n-3}k}E_{2^{n-1}+1,2^{n-1}+2},$$
$$\alpha(x^{k}[a,b]x^{-k})=I+t^{-2^{n-3}k}E_{1,2^{n-1}+2}.$$
Applying Lemma \ref{gh1m} where the assumptions are easily verified because of the alternating construction of $\alpha(a)$ and $\alpha(xax\inv)$ and symmetricity of formula in the lemma, we see that $\alpha(x^{k}ax^{-k})$'s commute with each other, so do $\alpha(x^{k}bx^{-k})$'s.  Moreover, we see that $\alpha(x^{k}a^2x^{-k})$'s commute with $\alpha(x^{k}bx^{-k})$'s and 
$$[\alpha(x^{2k}ax^{-2k}),\alpha(x^{2l}bx^{-2l})]=I+t^{-2^{n-3}(k+l)}E_{1,2^{n-1}+2}=\alpha(x^{k+l}[a,b]x^{-k-l}),$$
$$[\alpha(x^{2k+1}ax^{-2k-1}),\alpha(x^{2l}bx^{-2l})]=I+t^{-2^{n-3}(k+l+1)}E_{1,2^{n-1}+2}=\alpha(x^{k+l+1}[a,b]x^{-k-l-1}),$$
$$[\alpha(x^{2k}ax^{-2k}),\alpha(x^{2l+1}bx^{-2l-1})]=I+t^{-2^{n-3}(k+l+1)}E_{1,2^{n-1}+2}=\alpha(x^{k+l+1}[a,b]x^{-k-l-1}),$$
$$[\alpha(x^{2k+1}ax^{-2k-1}),\alpha(x^{2l+1}bx^{-2l-1})]=I+t^{-2^{n-3}(k+l+1)}E_{1,2^{n-1}+2}=\alpha(x^{k+l+1}[a,b]x^{-k-l-1}).$$
Therefore, the compatibility of $\alpha$ with the relations in Theorem \ref{pres} follows from above calculations and the commutator identities.

Finally, we show the homomorphism $\alpha$ has trivial kernel. Note that any torsion free element does not lie in the kernel. Let $\gamma$ be a torsion element in $\GG_2(\CC(M_{2^n}))$. It has the normal form as in Lemma \ref{m2nform}. Suppose $\alpha(\gamma)=I$. 
$\alpha(\ga)_{1,2}=0$ implies that there are no even conjugates of $a^{2k-1}, 1\leq k\leq 2^{n-3},$ in the normal form of $\gamma$.  $\alpha(\ga)_{1,3}=0$ implies that there are no odd conjugates of $a^{2k-1}, 1\leq k\leq 2^{n-3},$ in $\gamma$. Moreover, $\alpha(\ga)_{1,2^m}=0, 2\leq m\leq n-2,$ imply that there are no even conjugates of $a^{2k-2}, 2\leq k\leq 2^{n-3},$ in $\gamma$ and $\alpha(\ga)_{1,2^m+1}=0, 2\leq m\leq n-2,$ implies that there are no odd conjugates of $a^{2k-2}, 2\leq k\leq 2^{n-3},$ in $\gamma$.  $\alpha(\ga)_{2^{n-1},2^{n-1}+2}=\alpha(\ga)_{2^{n-1}+1,2^{n-1}+2}=0,$ implies that there are no conjugates of $b$ in $\gamma$ and  $\alpha(\ga)_{1,2^{n-1}+2}=0$ implies that there are no conjugates of $[a,b]$ in $\gamma$. Thus, $\alpha(\gamma)=I$ implies that $\gamma$ is trivial.

\end{proof}

\noindent{\bf Remark:} M. Larsen asked whether or not $F$ being a finite $p$-group is a necessary and sufficient condition for $\GG(\CC(F))$ to be linear over one field. We only know it is necessary. Ju. E. Vapn\`{e} \cite{V1,V2} and, independently, B. A. F. Wehrfritz \cite{W0, Wp} give necessary and sufficient conditions for a lamplighter group to be linear over some field. In particular, $\Z_p\wr\Z$ is only linear over characteristic $p$. Since $\GG(\CC(F))$ contains copies of $\Z_p\wr\Z$, for different $p$'s, then $F$ being a finite $p$-group is a necessary condition for $\GG(\CC(F))$ to be linear over some field. But this of course does not rule out the cases of embedding into a product of general linear groups over different fields.

\medskip
\noindent
Ning Yang\\
Department of Mathematics\\
Indiana University\\
Bloomington, IN 47405, USA\\
E-mail: ningyang@indiana.edu


\medskip




\begin{thebibliography}{}


\bibitem{BGS}
L. Bartholdi, R. I. Grigorchuk, and Z. \v{S}uni\'{k}, \emph{Branch groups} in: Handbook of Algebra, Vol. 3, 989-1112, North-Holland, Amsterdam, 2003.



\bibitem{CFK}
Y. de Cornulier, D. Fisher and N. Kashyap, \emph{Cross-wired lamplighter groups}, New York J. of Math. 18 (2012) 667--677.

\bibitem{DL}
R. Diestel and I. Leader,
\emph{ A conjecture concerning a limit of non-Cayley graphs}, J. Algebraic Combin. 14 (2001), no. 1, 17--25.


\bibitem{Eilenberg}
S. Eilenberg, \emph{Automata, Languages and Machines}, Academic Press, New York, Vol. A, 1974; Vol. B, 1976.


\bibitem{E}
A. Erschler, \emph{Boundary behavior for groups of subexponential growth}, Ann. of Math. (2) 160 (2004), no.3, 1183--1210.

\bibitem{EFW1}
A. Eskin, D. Fisher and K. Whyte, \emph{Quasi-isometries and rigidity
of solvable groups}, Pure Appl. Math. Q. 3 (2007), no. 4, part 1, 927--947.

\bibitem{EFW2}
A. Eskin, D. Fisher and K. Whyte, \emph{ Coarse differentiation of
quasi-isometries I: spaces not quasi-isometric to Cayley graphs}, Ann. of Math. (2) 176 (2012), no.1, 221--260.

\bibitem{EFW3}
A. Eskin, D. Fisher and K. Whyte, \emph{Coarse differentiation of
quasi-isometries II: rigidity for Sol and lamplighter groups}, Ann. of Math. (2) 177 (2013), no. 3, 869--910.


\bibitem{G}
R. I. Grigorchuk, \emph{Degrees of growth of finitely generated groups, and the theory of invariant means}, Math. USSR Izv. 25 (1985), 259--300.


\bibitem{GNS}
R. I. Grigorchuk, V. V. Nekrashevich and V.~I. Sushchanskii,
\textit{Automata, dynamical systems, and groups}, in: R.~I.
Grigorchuk, (ed.), ``Dynamical systems, automata, and infinite
groups.'' Proc. Steklov Inst. Math.
\textbf{231} (2000), 128--203; translation from Tr. Mat. Inst.
Steklova \textbf{231} (2000), 134--214.



\bibitem{GZ}
R. I. Grigorchuk and A. \.{Z}uk, \emph{The lamplighter group as a group generated by a 2-state automaton, and its spectrum}, Geom. Dedicata 87 (2001), 209--244.



\bibitem{Gru} 
K. W. Gruenberg, \emph{Residual properties of infinite soluble groups},  Proc.
London Math. Soc. (3) 7, 1957, 29--62.


\bibitem{KSS}
M. Kambites, P. V. Silva and B. Steinberg, \emph{The spectra of lamplighter groups and Cayley machines}, Geom. Dedicata 120 (2006), 193--227.

\bibitem{Moeller}
R. Moeller, Private correspondence to W. Woess, 2001.

\bibitem{Rhodes}
J. Rhodes, \emph{Monoids acting on trees: elliptic and wreath products and the holonomy theorem for arbitrary monoids with applications to infinite groups}, Internat. J. Algebra Comput. 1 (1991), 253--279.

\bibitem{SiSt} P. V. Silva and B. Steinberg, \emph{On a class of automata groups generalizing lamplighter groups}, Internat. J. Algebra Comput. 15 (2005), 1213--1235.

\bibitem{V1}
Ju. E. Vapn\`{e}, \emph{The representability of a direct wreath product of groups by matrices}, (Russian) Mat. Zametki 7 (1970), 181--190; (English translation) Math. Notes 7 (1970), 110--114.

\bibitem{V2}
Ju. E. Vapn\`{e}, \emph{A criterion for the representability of a direct wreath product of groups by matrices}, (Russian) Dokl. Akad. Nauk SSSR 195 (1970), 13--16; (English translation) Soviet Math. Dokl. 11 (1970), 1396--1399.

\bibitem{Wp}
B. A. F. Wehrfritz, \emph{Wreath products of linear groups; the characteristic-p case}, Bull. London Math. Soc. 3 (1971), 331--332.

\bibitem{W0}
B. A. F. Wehrfritz, \emph{Wreath products and chief factors of linear groups}, J. London Math. Soc. (2) 4 (1972), 671--681.


\bibitem{Woess}
W. Woess, \emph{Lamplighters, Diestel-Leader graphs, random walks,
and harmonic functions}, Combin. Probab.  Comput. 14
(2005) 415--433.

\bibitem{Wortman}
K. Wortman, \emph{ A Finitely-presented solvable group with a small
quasi-isometry group}, Michigan Math. J. 55 (2007), 3--24.


\end{thebibliography}
\end{document}